\documentclass[12pt]{msml2020}
\usepackage{times}
\usepackage{array}
\usepackage[capitalize,nameinlink]{cleveref} \crefname{equation}{}{Equations}
\Crefname{equation}{Eq.}{Equations}

\graphicspath{{images/}}

\newcommand\bbR{\mathbb{R}}

\newcommand\bbN{\mathbb{N}}

\newcommand{\tu}{\tilde{u}}
\newcommand{\tm}{\tilde{m}}
\newcommand\unit[1]{\widehat{#1}}
\newcommand\mU{\mathcal{U}}

\newcommand\dd{\,\mathrm{d}}     
\newcommand\sffont[1]{{\sf{#1}}}
\newcommand\relu{{\sffont{ReLU}}}
\newcommand\id{{\sffont{id}}}
\newcommand\Convone{{\sffont{Conv1d }}}
\newcommand\Convtwo{{\sffont{Conv2d }}}
\newcommand\BCR{{\text{\sffont{BCR-Net}}}}
\newcommand\NN{\mathrm{NN}}
\newcommand\T{\mathsf{T}}
\newcommand\C{\mathcal{C}}
\newcommand\cnn{\mathrm{cnn}}

\title[DL for Traveltime Tomography]{Solving Traveltime Tomography with Deep Learning}

\msmlauthor{ \Name{Yuwei Fan} \Email{ywfan@stanford.edu}\\
 \addr Department of Mathematics, Stanford University, Stanford, CA 94305. 
 \AND
 \Name{Lexing Ying} \Email{lexing@stanford.edu}\\
 \addr Department of Mathematics and ICME, Stanford University, Stanford, CA 94305.
}

\begin{document}
\maketitle
\begin{abstract}  This paper introduces a neural network approach for solving two-dimensional traveltime tomography
  (TT) problems based on the eikonal equation. The mathematical problem of TT is to recover the
  slowness field of a medium based on the boundary measurement of the traveltimes of waves going
  through the medium. This inverse map is high-dimensional and nonlinear. For the circular
  tomography geometry, a perturbative analysis shows that the forward map can be approximated by a
  vectorized convolution operator in the angular direction. Motivated by this and filtered
  back-projection, we propose an effective neural network architecture for the inverse map using the
  recently proposed BCR-Net, with weights learned from training datasets. Numerical results
  demonstrate the efficiency of the proposed neural networks.
\end{abstract}


\section{Introduction}\label{sec:intro}
Traveltime tomography is a method to determinate the internal properties of a medium by measuring
the traveltimes of waves going through the medium. It is first motivated in global seismology in
determining the inner structure of the Earth by measuring at different seismic stations the
traveltimes of seismic waves produced by earthquakes
\citep{backus1968resolving,rawlinson2010seismic}. By now, it has found many applications, such as
Sun's interior \citep{kosovichev1996tomographic}, ocean acoustics \citep{munk2009ocean}, and
ultrasound tomography \citep{schomberg1978improved, jin2006thermoacoustic} in biomedical imaging.

\paragraph{Background.}
The governing equation of first-arrival traveltime tomography (TT) is the eikonal equation
\citep{born1965principles} and we consider the two dimensional case in this paper for
simplicity. Let $\Omega\subset\bbR^2$ be an open bounded domain with Lipschitz boundary
$\Gamma=\partial\Omega$. Suppose that the positive function $m(x)$ is the slowness field, i.e., the
reciprocal of the velocity field, defined in $\Omega$. The traveltime $u(x)$ satisfies the eikonal
equation $|\nabla u(x)| = m(x)$.
Since it is a special case of the Hamilton-Jacobi equation, the solution $u(x)$ can develop
singularities and should be understood in the viscosity sense \citep{ishii1987simple}.

A typical experimental setup of TT is as follows. For each point $x_s\in\Gamma$, one sets up the
Soner boundary condition at point $x_s$, i.e., only zero value at $x_s$, and solves for the
following the eikonal equation
\begin{equation}\label{eq:eikonal}
  \begin{aligned}
    |\nabla u^s(x)| &= m(x), \quad x\in\Omega, \\
    u^s(x_s) &= 0,
  \end{aligned}
\end{equation}
where the superscript $s$ is to index the source point. Recording the solution of $u^s(\cdot)$ at
points $x_r\in\Gamma$ produces the whole data set $u^s(x_r)$. In practice $x_r$ and $x_s$ are
samples from a discrete set of points on $\Gamma$. Here we assume for now that they are placed
everywhere on $\Gamma$, for the simplicity of presentation and mathematical analysis.

The forward problem is to compute $u^s(x_r)$ given the slowness field $m(x)$. On the other hand, the
inverse problem, at the center of the first-arrival TT, is to recover $m(x)$ given $u^s(x_r)$.

Both the forward and inverse problems are computationally challenging, and a lot of efforts have been
devoted to their numerical solutions. For the forward problem, the eikonal equation, as a special
case of the Hamilton-Jacobi equation, can develop singular solutions. In order to compute the
physically meaningful viscosity solution, special care such as up-winding is required. As the
resulting discrete system is nonlinear, fast iteration methods such as fast marching method
\citep{popovici1997three,sethian1999fast} and fast sweeping method
\citep{zhao2005fast,kao2005fast,qian2007fast} have been developed. Among them, the fast sweeping
methods have been successfully applied to many traveltime tomography problems
\citep{leung2006adjoint}. The inverse problem is often computationally more intensive, due to the
nonlinearity of the problem. Typical methods take an optimization approach with proper
regularization \citep{chung2011adaptive} and require a significant number of iterations.

\paragraph{A deep learning approach.}
Over the past decade or so, deep learning (DL) has become the dominant approach in computer vision,
image processing, speech recognition, and many other applications in machine learning and data
science
\citep{Hinton2012,Krizhevsky2012,goodfellow2016deep,MaSheridan2015,Leung2014,SutskeverNIPS2014,leCunn2015,SCHMIDHUBER2015}.
From a technical point of view, this success is a synergy of several key developments: neural
networks (NNs) as a flexible framework for representing high-dimensional functions and maps, simple
algorithms such as back-propagation (BP) and stochastic gradient descent (SGD) for tuning the model
parameters, efficient general software packages such as Tensorflow and Pytorch, and unprecedented
computing power provided by GPUs and TPUs.

In the past several years, deep neural networks (DNNs) have been increasingly used in scientific
computing, particularly in solving PDE-related problems
\citep{khoo2017solving,berg2017unified,han2018solving,fan2018mnn,Araya-Polo2018,Raissi2018,kutyniok2019theoretical,feliu2019meta},
in two directions. In the first direction, as NNs offer a powerful tool for approximating
high-dimensional functions \citep{cybenko1989approximation}, it is natural to use them as an ansatz
for high-dimensional PDEs
\citep{rudd2015constrained,carleo2017solving,han2018solving,khoo2019committor, weinan2018deep}. The
second direction focuses on the low-dimensional parameterized PDE problems, by using the DNNs to
represent the nonlinear map from the high-dimensional parameters of the PDE solution
\citep{long2018pde,han2017deep,khoo2017solving,fan2018mnn,fan2018mnnh2,fan2019bcr,li2019variational,bar2019unsupervised}.

As an extension of the second direction, DNNs have been widely applied to inverse problems
\citep{khoo2018switchnet,hoole1993artificial,kabir2008neural,adler2017solving,lucas2018using,tan2018image,fan2019eit,fan2019ot,raissi2019physics}. For
the forward problem, since applying neural networks to input data can be carried out rapidly under
current software and hardware architectures, the solution of the forward problem can be
significantly accelerated when the forward map is represented with a DNN. For the inverse problem,
DNNs can help in two critical ways: (1) due to its flexibility in representing high-dimensional
functions, DNNs can potentially be used to approximate the full inverse map, thus avoiding the
iterative solution process; (2) recent work in machine learning shows that DNNs often can
automatically extract features from the data and offer a data-driven regularization prior.

This paper applies the deep learning approach to the first-arrival TT by representing the whole
inverse map using a NN. The starting point is a perturbative analysis of the forward map, which
reveals that for the circular tomography geometry, the forward map contains a one-dimensional
convolution with multiple channels, after appropriate reparameterization. This observation motivates
to represent the forward map from 2D coefficient $m(x)$ to 2d data $u^s(x_r)$ by a
\emph{one-dimensional} convolution neural network (with multiple channels).  Further, the
one-dimensional convolution neural network can be implemented by the recently proposed multiscale
neural network \citep{fan2018mnn,fan2019bcr}.  Following the idea of filtered back-projection
\citep{schuster1993wavepath}, the inverse map can be approximated by the adjoint map followed by a
pseudo-differential filtering step. This suggests an architecture for the inverse map by reversing
the architecture of the forward map followed with a simple two-dimensional convolution neural
network.

For the test problems being considered, the resulting neural networks have $10^5$ parameters
when the data is of size $160\times 160$ (a fully-connected layer results in $160^4\approx 6\times
10^8$ parameters), thanks to the convolutional structure and the compact multiscale neural network. This rather small
number of parameters allows for rapid and accurate training, even on rather limited data sets.

\paragraph{Organization.}
This rest of the paper is organized as follows. The mathematical background is given in
\cref{sec:math}. The design and architecture of the DNNs of the forward and inverse maps are
discussed in \cref{sec:nn}. The numerical results in \cref{sec:num} demonstrate the numerical
efficiency and the generalization of the proposed neural networks.

\section{Mathematical analysis of traveltime tomography}\label{sec:math}

\subsection{Problem setup}

This section describes the necessary mathematical insights that motivate the NN architecture design.
Let us consider the so-called differential imaging setting, where a background slowness field
$m_0(x)$ is known, and denote by $u_0^s$ the solution of the eikonal equation associated with the
field $m_0$:
\begin{equation}\label{eq:bg}
  \begin{aligned}
    |\nabla u_0^s(x)| &= m_0(x), \quad x\in\Omega, \\
    u_0^s(x_s) &= 0.
  \end{aligned}
\end{equation}
Then for a perturbation $\tm$ to the slowness field, the difference in the traveltime $\tu^s \equiv
u^s - u^s_0$ naturally satisfies
\begin{equation}\label{eq:perturbation}
  \begin{aligned}
    |\nabla (u_0^s(x) + \tu^s(x))| &= m_0(x) + \tm(x), \quad x\in\Omega, \\
    \tu^s(x_s) &= 0.
  \end{aligned}
\end{equation}
The imaging data $d(x_s,x_r)$ consists of $\tu^s(x_r)$ over all $x_s$ and $x_r$: $d(x_s,x_r) \equiv
\tu^s(x_r)$.

To better understand the dependence of $\tu^s$ on $\tm$, we assume $\tm$ to be sufficient small and
carry out a perturbative analysis. Squaring \cref{eq:perturbation} and canceling the background
using \cref{eq:bg} result in 
\begin{equation}
  (\nabla \tu^s(x))^\T\nabla \tu^s(x) + 2(\nabla u_0^s(x))^\T\nabla \tu^s(x)
  = \tm(x)^2 + 2 m_0(x)\tm(x).
  \end{equation}
Since $\tm(x)$ is sufficiently small, $\nabla \tu^s(x)$ is also a small quantity. Keeping only
linear terms in $\tm$ and discarding the higher order ones yields
\begin{equation}
  \nabla u_0^s(x)^\T \nabla \tu^s(x) \approx m_0(x) \tm(x),
\end{equation}
which is an advection equation. Using $|\nabla u_0^s(x)^\T|=m_0(x)$, one can further simplify the
upper equation as
\begin{equation}\label{eq:approximation}
  \unit{\nabla u_0^s(x)}^\T \nabla \tu^s(x) \approx \tm(x),
\end{equation}
where $\unit{\cdot}$ stands for the unit vector.

For simplicity, let $\C_0(x_s,x_r)$ be the unique characteristic of $u_0^s(x_r)$ that connects $x_s$
and $x_r$. Then
\begin{equation}\label{eq:d1}
  d(x_s,x_r) \equiv \tu^s(x_r) \approx \int_{\C_0(x_s,x_r)} \tm(x) \dd x \equiv d_1(x_s,x_r),
\end{equation}
where $d_1(x_s,x_r)$ is introduced to stand for the first-order approximation to $d(x_s,x_r)$.
Particularly, if the background slowness field is a constant, then $\C_0(x_s, x_r)$ is a line
segment with start and end points to be $x_s$ and $x_r$, respectively, and
\begin{equation*}
  d_1(x_s, x_r) = |x_s-x_r|\int_{0}^1 \tm(x_s+\tau(x_r-x_s))\dd\tau.
\end{equation*}

\begin{figure}[htb]
  \centering
  \includegraphics[width=0.30\textwidth]{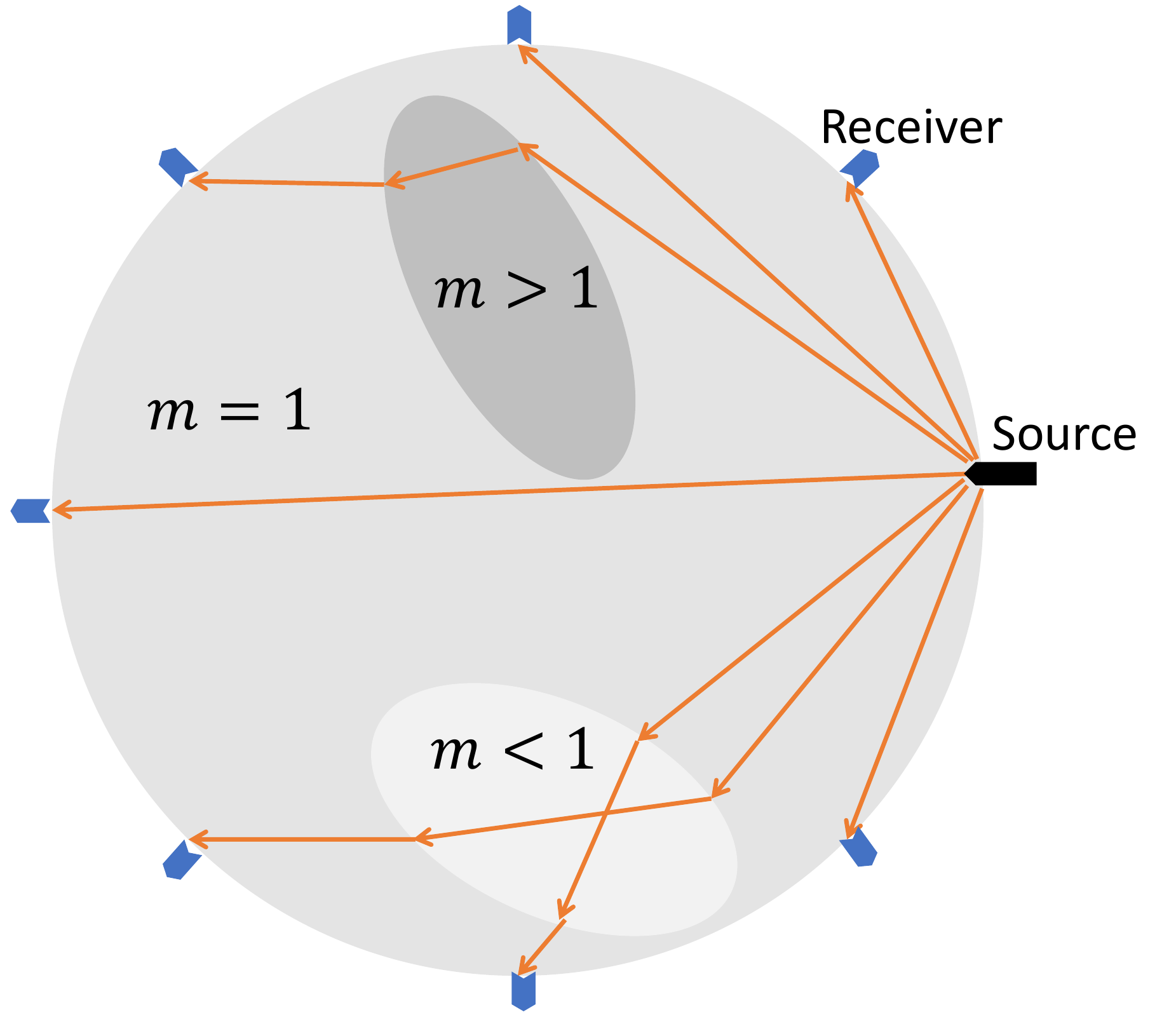}
  \caption{\label{fig:domain} Illustration of the problem setup. The domain is a unit disk and the
  light sources and the receivers are equidistantly placed on the boundary.}
\end{figure}

The most relevant geometry in traveltime tomography either for medicine and earth science is the
circular geometry where $\Omega$ is modeled as a unit disk
\citep{chung2011adaptive,deckelnick2011numerical,yeung2018numerical}. As illustrated in
\cref{fig:domain}, the sources and receivers are placed on the boundary equidistantly. More
precisely, $x_s=(\cos(s),\sin(s))$ with $s=\frac{2\pi k}{N_s}$, $k=0,\dots,N_s-1$ and
$x_r=(\cos(r),\sin(r))$ with $r=\frac{2\pi j}{N_r}$, $j=0,\dots, N_r-1$, where $N_s=N_r$ in the
current setup.

Often in many cases, the background slowness field $m_0(x)$ is only radially dependent, or even a
constant \citep{deckelnick2011numerical, yeung2018numerical}. In what follows, $m_0(x)$ is assumed
to be radially dependent, i.e., $m_0(x) = m_0(|x|)$.

\subsection{Mathematical analysis on the forward map}\label{sec:analysis}

Since the domain $\Omega$ is a disk, it is convenient to rewrite the problem in the polar
coordinates.  Let $x_r=(\cos(r), \sin(r))$, $x_s=(\cos(s), \sin(s))$ and
$x=(\rho\cos(\theta),\rho\sin(\theta))$, where $\rho\in[0,1]$ is the radial coordinate and
$r,s,\theta\in[0, 2\pi)$ are the angular ones..

\begin{figure}[ht]
  \centering
  \begin{tabular}{c@{}c@{}c}
    \subfigure[$\tm(x)$]{
      \includegraphics[width=0.2\textwidth]{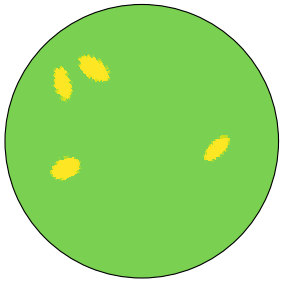} 
    } &
    \subfigure[$u^s(x_r)$]{
      \includegraphics[width=0.3\textwidth]{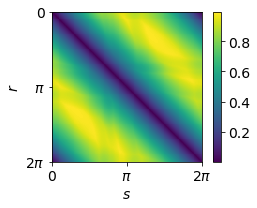} 
    } &
    \subfigure[$d(x_s, x_r)$]{
      \includegraphics[width=0.32\textwidth]{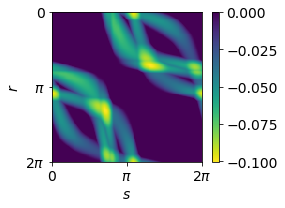}
    }\\
    \subfigure[$\tm(\theta, \rho)$]{
      \includegraphics[width=0.3\textwidth]{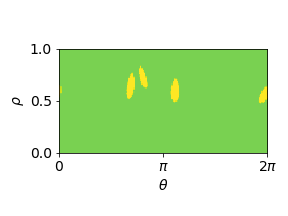} 
    } &
    \subfigure[$u^s(x_{s+h})$]{
      \includegraphics[width=0.3\textwidth]{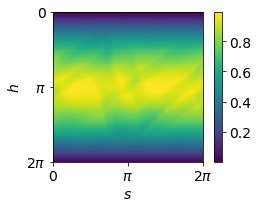}
    } &
    \subfigure[$d(s, h)$]{
      \includegraphics[width=0.32\textwidth]{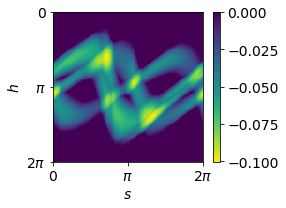}
    }
  \end{tabular}
  \caption{\label{fig:measurement} Visualization of the slowness field and the measurement data.
    The upper figures are the perturbation of the slowness field $\tm(x)$ ($m_0=1$ and $\tm\leq0$ in
    this sample), the measurement data $u^s(x_r)$ and the difference $d(x_s, x_r)$ with respect to
    the background measurement data.  The lower-left figure is $\tm(x)$ in the polar coordinates and
    the lower-right two figures the ``shear'' of their corresponding upper figures.
  }
\end{figure}

\Cref{fig:measurement} presents an example of the slowness field and the measurement data. Notice
that the main signal in $u^s(x_r)$ and $d(x_s, x_r)$ concentrates on the minor diagonal part.  Due
to the circular tomography geometry, it is convenient to ``shear'' the measurement data by
introducing a new angular variable $h = r - s$, where the difference here is understood modulus
$2\pi$. As we shall see in the next section, this shearing step significantly simplifies the
architecture of the NNs. Under the new parameterization, the measurement data is
\begin{equation}
    d(s, h) \equiv d(x_s, x_{s+h}).
\end{equation}
The same convention applies to its first order approximation: $d_1(s, h) \equiv
d_1(x_s,x_{s+h})$. By writing $\tm(\theta, \rho) \equiv \tm( \rho\cos(\theta), \rho\sin(\theta) )$
in the polar coordinates, the linear dependence of $d_1(s,h)$ on $\tm$ in \eqref{eq:d1} states that
there exists a kernel distribution $K(s, h, \theta, \rho)$ such that
\begin{equation}\label{eq:d1kernel}
  d_1(s, h) = \int_{0}^1\int_{0}^{2\pi}K(s, h, \theta, \rho) \tm(\theta, \rho)\dd\rho\dd\theta.
\end{equation}

\paragraph{Convolution form of the map $m(\theta,\rho)\to d_1(s,h)$.}
Since the domain is a disk and $m_0$ is only radially independent, the whole problem is
equivariant to rotation. In this case, the situation can be dramatically simplified. Precisely, 
we have the following proposition.
\begin{proposition}\label{pro:convolution}
  There exists a function $\kappa(h, \rho, \cdot)$ periodic in the last parameter such that
  \begin{equation}\label{eq:convolution}
    d_1(s, h) = \int_0^1 \int_0^{2\pi} \kappa(h,\rho,s-\theta) \tm(\theta,\rho) \dd\rho\dd\theta.
  \end{equation}
\end{proposition}
\begin{proof}
Let $\C_0(s, r) \equiv \C_0((\cos(s),\sin(s)), (\cos(r),\sin(r)))$ and we parameterize the
characteristic $\C_0(s, r)$ as $p_{s,h}(\tau) \equiv (\theta_{s,h}(\tau), \rho_{s,h}(\tau))$,
$\tau\in[0,1]$ with $\rho_{s,h}(0)=\rho_{s,h}(1)=1$ and $\theta_{s,h}(0)=s$, $\theta_{s,h}(1)=r$.
Then the relationship \eqref{eq:d1} between $d_1$ and $\tm$ can be written as 
\[
  d_1(s, h) = \int_0^1\tm(\theta_{s,h}(\tau),\rho_{s,h}(\tau))\|p'_{s,h}(\tau)\|\dd\tau.
\]

Since the background slowness $m_0$ is radially independent, the characteristic $\C_0(s, r)$
is rotation invariant in the sense that for any $\phi\in[0,2\pi)$, if
$(\theta_{s,h}(\tau),\rho_{s,h}(\tau))$ is a parameterization of $\C_0(s,r)$, then
$(\theta_{s,h}(\tau)+\phi, \rho_{s,h}(\tau))$ is a parameterization of $\C_0(s+\phi,r+\phi)$.
Hence, for any $\phi\in[0,2\pi)$, if we rotate the system by a angular $\phi$, then 
\[
  \begin{aligned}
    d_1(s+\phi, h) &= \int_0^1\tm(\theta_{s+\phi,h}(\tau),\rho_{s+\phi,h}(\tau))\|p'_{s+\phi,h}(\tau)\|\dd\tau\\
    & = \int_0^1\tm(\theta_{s,h}(\tau)+\phi,\rho_{s,h}(\tau))\|p'_{s,h}(\tau)\|\dd\tau.
  \end{aligned}
\]
Writing this equation in the form of \eqref{eq:d1kernel} directly yields $K(s+\phi, h, \theta+\phi,
\rho) = K(s, h,\theta,\rho)$. Hence, there is a periodic $\kappa(s, h,\cdot)$ in the last parameter
such that $K(s, h,\theta,\rho)=\kappa(h, \rho, s-\theta)$. This completes the proof.
\end{proof}

Proposition \ref{pro:convolution} shows that $K$ acts on $\tm$ in the angular direction by a convolution, which
is, in fact, the motivation behind shearing the measurement data $d$.  This property allows us to
evaluate the map $\tm(\theta,\rho)\to d(s,h)$ by a family of 1D convolutions, parameterized $\rho$
and $h$.

\paragraph{Discretization.}
All the above analysis is in the continuous space. One can apply a discretization on the eikonal
equation \cref{eq:eikonal} by finite difference and solve it by fast sweeping method or fast
marching method. Here we assume that the discretization of $\tm(\theta,\rho)$ is on a uniform mesh
on $[0,2\pi)\times[0, 1]$.
More details of the discretization and the numerical solver will be discussed in the \cref{sec:num}. 
With a slight abuse of notation, we use the same letters to denote the continuous kernels, variables
and their discretization. Then the discretization version of \cref{eq:d1kernel,eq:convolution} is
\begin{equation}\label{eq:discrete}
  d(s, h) \approx \sum_{\rho,\theta}K(s, h, \theta,\rho)\tm(\theta, \rho) 
  = \sum_{\rho}(\kappa(h,\rho, \cdot) * \tm(\cdot, \rho))(s).
\end{equation}

\section{Neural networks for TT}\label{sec:nn}
In this section, we describe the NN architecture for the inverse map $d(s, h)\to \tm(\theta, \rho)$
based on the mathematical analysis in \cref{sec:math}. To start, we first study the NN for the
forward map and then the inverse map.

\paragraph{Forward map.}
The perturbative analysis in \cref{sec:analysis} shows that, when $\tm$ is sufficiently small, the
forward map $\tm(\theta,\rho)\to d(s, h)$ can be approximated by \cref{eq:discrete}. In terms of the
NN architecture, for small $\tm$, the forward map \cref{eq:discrete} can be approximated by a
(non-local) convolution on the angular direction and a fully-connected operator on the $(h, \rho)$
direction.  In the actual implementation, it can be represented by the convolution layer by taking
$h$ and $\rho$ as the channel dimensions. For larger $\tm$, this linear approximation is no longer
accurate. In order to extend the neural network for \cref{eq:discrete} to the nonlinear case, we
propose to increase the number of convolution layers and nonlinear activation functions.

\begin{algorithm}[ht]
  \SetAlgoLined
  \KwData{$c$, $N_{\cnn}\in\bbN^+$, $\tm\in\bbR^{N_{\theta}\times N_{\rho}}$}
  \KwResult{$d \in\bbR^{N_s\times N_h}$}
  \tcc{Resampling data to fit for \BCR.}
  $\xi = \Convone[c,1,\id](\tm)$ with $\rho$ as the channel direction

  \tcc{Use {\BCR} to implement the convolutional neural network.}
  $\zeta = \BCR[c, N_{\cnn}](\xi)$

  \tcc{Reconstruct the result from the output of \BCR.}
  $d = \Convone[N_{h}, 1, \id](\zeta)$
  \caption{\label{alg:forward} Neural network architecture for the forward map $\tm\to d$.}
\end{algorithm}

In the $(h, \rho)$ direction, denote the number of channels by $c$, whose value is problem-dependent
and will be discussed in the numerical part. In the angular direction, since the convolution between
$\tm$ and $d$ is global, in order to represent global interactions the window size of the
convolution $w$ must satisfy the following relationship
\begin{equation}
  w N_{\cnn}\geq N_{\theta},
\end{equation}
where $N_{\cnn}$ is the number of layers and $N_{\theta}$ is number of discretization points on the
angular direction. A simple calculation shows that the number of parameters of the neural network is
$O(wN_{\cnn}c^2)\sim O(N_{\theta}c^2)$. The recently proposed {\BCR} \citep{fan2019bcr} has been
demonstrated to require fewer number of parameters and provide better efficiency for such global
interactions. Therefore, in our architecture, we replace the convolution layers with the
{\BCR}. The resulting neural network architecture for the forward map is summarized in
\cref{alg:forward} with an estimate of $O(c^2\log(N_{\theta})N_{\cnn})$ parameters. The components
are explained in the following.
\begin{itemize}
\item $\xi=\Convone[c, w, \phi](m)$ mapping $m\in \bbR^{N_{\theta}\times N_{\rho}}$ to
  $\xi\in\bbR^{N_{\theta}\times c}$ is the one-dimensional convolution layer with window size $w$,
    channel number $c$, activation function $\phi$ and period padding on the first direction.
  \item {\BCR} is motivated by the data-sparse nonstandard wavelet representation of the
    pseudo-differential operators \citep{bcr}. It processes the information at different scale
    separately and each scale can be understood as a {\em local} convolutional neural network. The
    one-dimensional $\zeta=\BCR[c, N_{\cnn}](\xi)$ maps $\xi\in\bbR^{N_{\theta}\times c}$ to
    $\zeta\in\bbR^{N_{\theta}\times c}$ where the number of channels and layers in the local
    convolutional neural network in each scale are $c$ and $n_{\cnn}$, respectively. The readers are
    referred to \citep{fan2019bcr} for more details on the \BCR.
\end{itemize}

\paragraph{Inverse map.}
The perturbative analysis in \cref{sec:analysis} shows that if $\tm$ is sufficiently small, the
forward map can be approximated by $d \approx K \tm$, the operator notation of the discretization
\cref{eq:discrete}. Here $\tm$ is a vector indexed by $(\theta, \rho)$, $d$ is a vector indexed by
$(s, h)$, and $K$ is a matrix with row indexed by $(s, h)$ and column indexed by $(\theta, \rho)$.

The filtered back-projection method \citep{schuster1993wavepath} suggests the following formula to
recover $\tm$:
\begin{equation}
  \tm\approx (K^\T K + \epsilon I)^{-1} K^\T d,
\end{equation}
where $\epsilon$ is a regularization parameter. The first piece $K^\T d$ can also be written as a
family of convolutions
\begin{equation}
  (K^\T d)(\theta,\rho) = \sum_{h} (\kappa(h, \rho, \cdot) * d(\cdot, h))(\theta).
\end{equation}
The application of $K^\T$ to $d$ can be approximated with a similar neural network to $K$ in
\cref{alg:forward}. The second piece $(K^\T K + \epsilon I)^{-1}$ is a pseudo-differential operator
in the $(\theta, \rho)$ space and it is implemented with several two-dimensional convolutional
layers for simplicity.  Then the resulting architecture for the inverse map is summarized in
\cref{alg:inverse} and illustrated in \cref{fig:inverse}. The $\Convtwo[c_2, w, \phi]$ used in
\cref{alg:inverse} is the two-dimensional convolution layer with window size $w$, channel number
$c_2$, activation function $\phi$ and periodic padding on the first direction and zero padding on
the second direction. The selection of the hyper-parameters in \cref{alg:inverse} will be discussed
in \cref{sec:num}.

\begin{algorithm}[htb]
  \SetAlgoLined
  \KwData{$c, c_2$, $w$, $N_{\cnn}$, $N_{\cnn2}\in\bbN^+$, $d\in\bbR^{N_{s}\times N_h}$}
  \KwResult{$\tm \in\bbR^{N_{\theta}\times N_{\rho}}$}
  \tcc{Application of $K^\T$ to $d$}
  $\zeta = \Convone[c, 1, \id](d)$ with $h$ as the channel direction

  $\xi = \BCR[c, N_{\cnn}](\zeta)$

  $\xi^{(0)} = \Convone[N_\rho, 1, \id](\xi)$

  \tcc{Application of $(K^\T K+\epsilon I)^{-1}$}
  \For{$k$ from $1$ to $N_{\cnn2}-1$}{
    $\xi^{(k)} = \Convtwo[c_2, w, \relu](\xi^{(k-1)})$
  }
  $\tm = \Convtwo[1, w, \id](\xi^{(N_{\cnn2}-1)})$
  \caption{\label{alg:inverse} Neural network architecture for the inverse problem $d \to \tm$.}
\end{algorithm}

\begin{figure}[ht]
  \centering
  \includegraphics[width=0.9\textwidth]{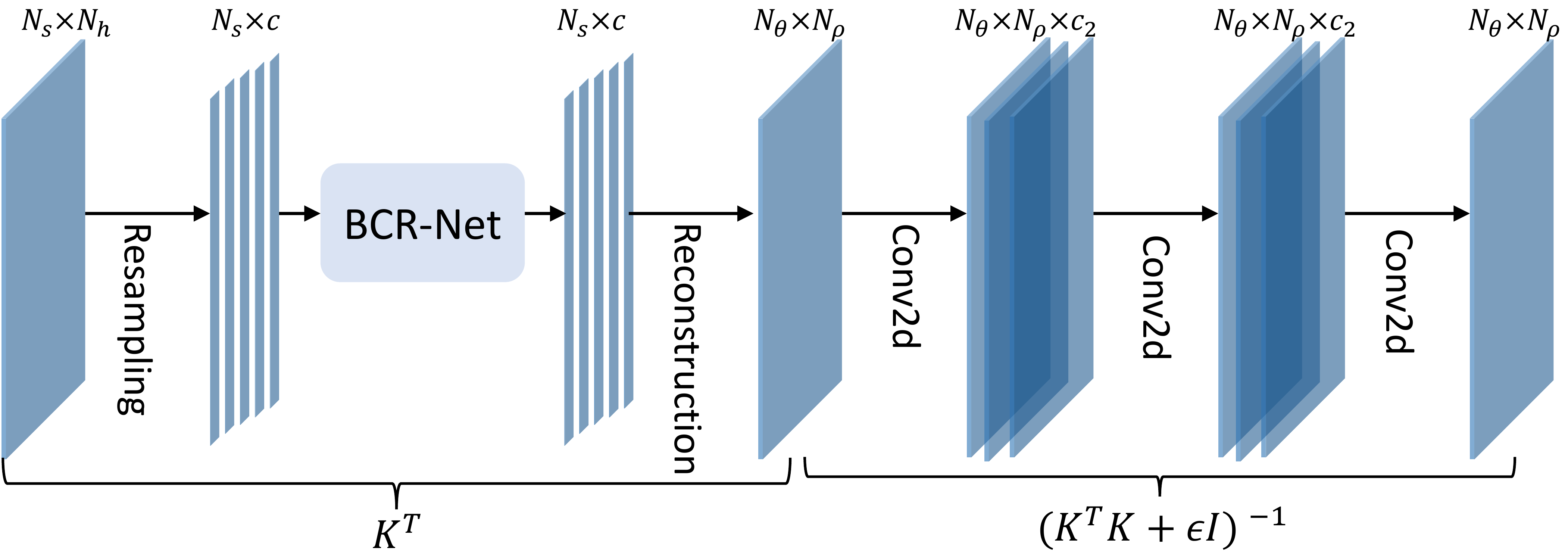}
  \caption{ \label{fig:inverse} Neural network architecture for the inverse map of TT.}
\end{figure}

\section{Numerical tests} \label{sec:num}
This section reports the numerical performance of the proposed neural network architecture in
\cref{alg:inverse} for the inverse map $d\to \tm$.

\subsection{Experimental setup}

In order to solve the eikonal equation \cref{eq:eikonal} on the unit disk $\Omega$, we embed
$\Omega$ into the square domain $[-1,1]^2$ by specifying sufficiently large slowness values
outside $\Omega$.  The domain $[-1,1]^2$ is discretized with a uniform Cartesian mesh with $160$
points in each direction by a finite difference scheme. The fast sweeping method proposed in
\citep{zhao2005fast} is used to solve the nonlinear discrete system. In the polar coordinates, the
domain $(\theta, \rho) \in [0,2\pi]\times [0,1]$ is partitioned by uniformly Cartesian mesh with
$160\times 80$ points, i.e., $N_\theta=160$ and $N_{\rho}=80$. As $\tm(\theta,\rho)$ used in
\cref{alg:inverse} is in the polar coordinates while the eikonal equation is solved in the Cartesian
ones, the perturbation of the slowness field $\tm$ is treated as a piecewise linear function in the
domain $\Omega$ and is interpolated on to the polar grid. The number of sources and receivers as
$N_s=N_r=160$, and hence $N_h=160$.

The NN in \cref{alg:inverse} is implemented with Keras \citep{keras} running on top of TensorFlow
\citep{tensorflow}. All the parameters of the network are initialized by Xavier initialization
\citep{glorot2010understanding}.  The loss function is the mean squared error, and the optimizer is
the Nadam \citep{dozat2015incorporating}.  In the training process, the batch size and the learning
rate is firstly set as $32$ and $10^{-3}$ respectively, and the NN is trained $100$ epochs. One then
increases the batch size by a factor $2$ till $512$ with the learning rate unchanged, and then
decreases the learning rate by a factor $10^{1/2}$ to $10^{-5}$ with the batch size fixed as
$512$. In each step, the NN is trained with $50$ epochs. For the hyper-parameters used in
\cref{alg:inverse}, $N_{\cnn}=6$, $N_{\cnn2}=5$, and $w=3\times 3$. The selection of the channel
number $c$ will be studied later.

\subsection{Results}
For a fixed $\tm$, $d(s, h)$ stands for the \emph{exact} measurement data solved by numerical
discretization of \cref{eq:eikonal}. The prediction of the NN from $d$ is denoted by $\tm^{\NN}$.
The metric for the prediction is the peak signal-to-noise ratio (PSNR), which is defined as
\begin{equation}\label{eq:psnr}
  \mathrm{PSNR} = 
  10 \log_{10}\left(\frac{\mathrm{Max}^2}{\mathrm{MSE}}\right),
  ~~\mathrm{Max} = \max_{ij}(\tm_{ij}) - \min_{ij}(\tm_{ij}),
  ~~\mathrm{MSE} = \frac{1}{N_{\theta}N_{\rho}}\sum_{i,j}|\tm_{i,j}-\tm_{i,j}^{\NN}|^2.
\end{equation}
For each experiment, the test PSNR is then obtained by averaging \cref{eq:psnr} over a given set of
test samples. The numerical results presented below are obtained by repeating the training process
five times, using different random seeds for the NN initialization.

The numerical experiments focus on the shape reconstruction setting
\citep{ustundag2008retrieving,deckelnick2011numerical}, where $\tm$ are often piecewise constant
inclusions. The background slowness field is set as $m_0\equiv 1$ and the slowness field $\tm$ is
assumed to be the sum of $N_e$ piecewise constant ellipses. As the slowness field $m$ is positive,
it is required that $\tm>-1$. For each ellipse, the direction is uniformly random
over the unit circle, the position is uniformly sampled in the disk, and the width and height
depend on the datasets.
It is also required that each ellipse lies in the disk and there is no intersection between every two
ellipses.  Three types of data sets are generated to test the neural network.
\begin{itemize}
  \item Negative inclusions. $\tm$, the perturbation of the slowness, is $-0.5$ in the ellipses and
    $0$ otherwise, and the width and height of each ellipse are sampled from the uniform
    distributions $\mU(0.1, 0.2)$ and $\mU(0.05, 0.1)$, respectively.
  \item Positive inclusions. $\tm$ is $2$ in the ellipses and $0$ otherwise, and 
      the width and height of each ellipse are sampled from $\mU(0.2,0.4)$ and $\mU(0.1, 0.2)$,
      respectively.
  \item Mixture inclusions. The setup of each ellipse is either a negative one in the negative
      inclusions or a positive one in the positive inclusions.
      \end{itemize}
For each type, we generate two datasets for $N_e=2$ and $4$.
For each test, $20,480$ samples $\{(\tm_i, d_i)\}$ are generated with $16,384$ used for training and
the remaining $4,096$ for testing.

\begin{figure}[h!]
  \centering
  \subfigure[Negative inclusions]{
    \includegraphics[width=0.3\textwidth]{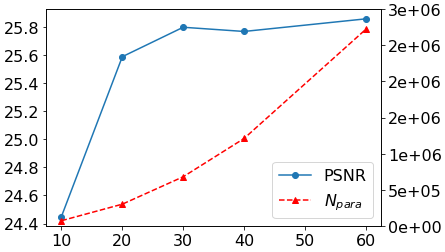}
  }
  \subfigure[Positive inclusions]{
    \includegraphics[width=0.3\textwidth]{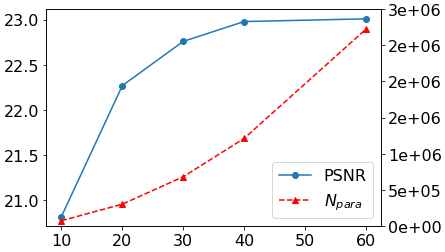}
  }
  \subfigure[Mixture inclusions]{
    \includegraphics[width=0.3\textwidth]{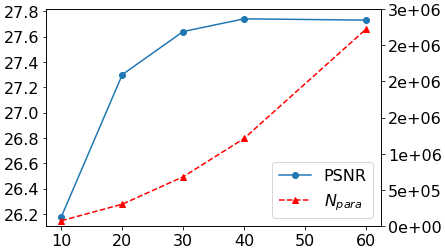}
  }
  \caption{\label{fig:trErr} The test PSNR for different channel numbers $c$ for the three types of
    data with $N_e=4$.}
\end{figure}

\begin{figure}[h!]
  \centering
  \begin{tabular}{l@{}cccc}
    & reference & $\tm^{\NN}$ with $\delta=0$ &
    $\tm^{\NN}$ with $\delta=2\%$ &
    $\tm^{\NN}$ with $\delta=10\%$\\
    \rotatebox{90}{\phantom{abcdefgh}Neg}\phantom{ab} &
    \includegraphics[width=0.22\textwidth]{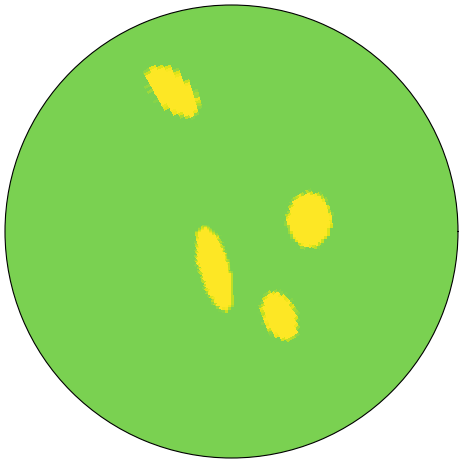}     &
    \includegraphics[width=0.22\textwidth]{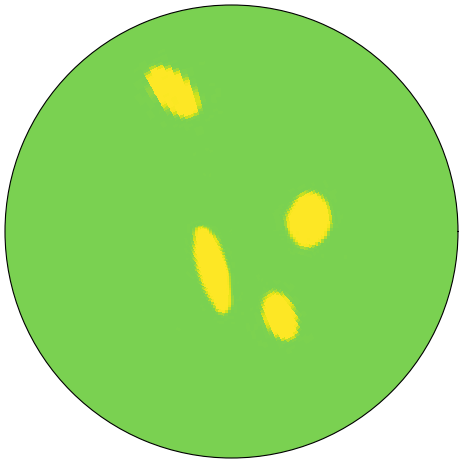} &
    \includegraphics[width=0.22\textwidth]{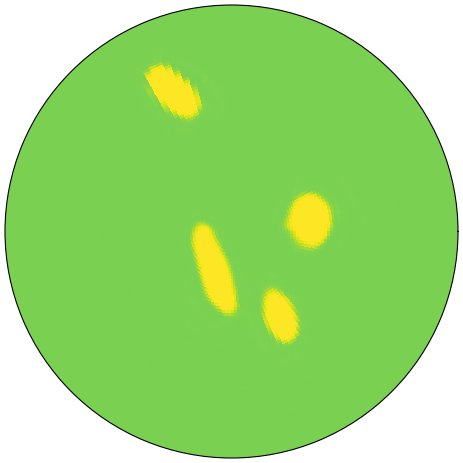} &
    \includegraphics[width=0.22\textwidth]{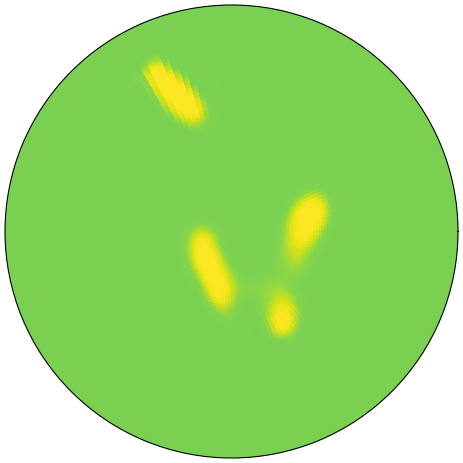} \\
    \rotatebox{90}{\phantom{abcdefgh}Pos}\phantom{ab} &
    \includegraphics[width=0.22\textwidth]{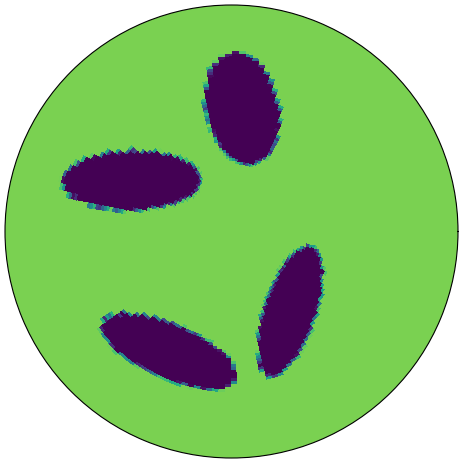} &
    \includegraphics[width=0.22\textwidth]{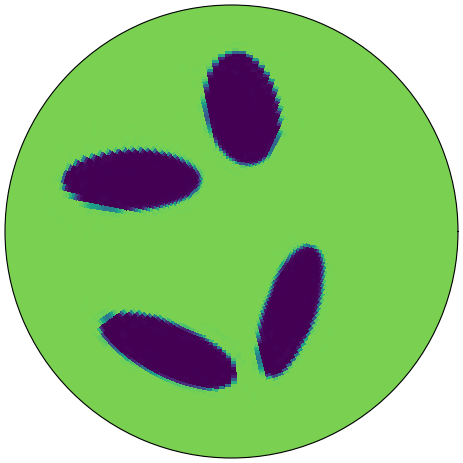} &
    \includegraphics[width=0.22\textwidth]{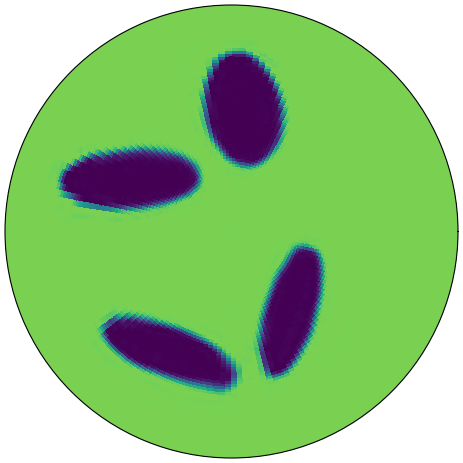} &
    \includegraphics[width=0.22\textwidth]{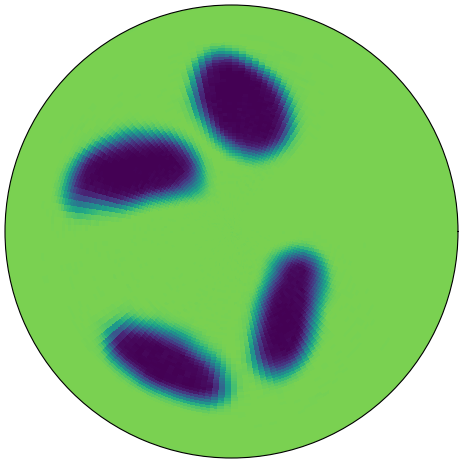} \\
    \rotatebox{90}{\phantom{abcdefgh}Mix}\phantom{ab} &
    \includegraphics[width=0.22\textwidth]{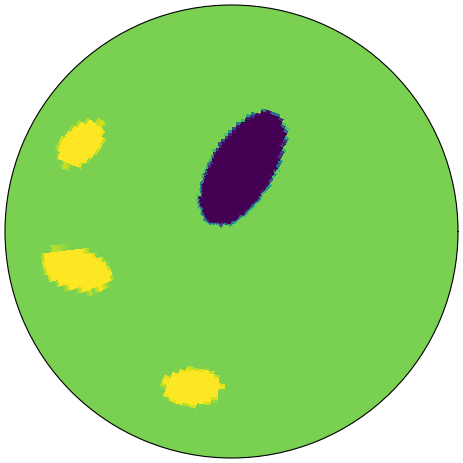} &
    \includegraphics[width=0.22\textwidth]{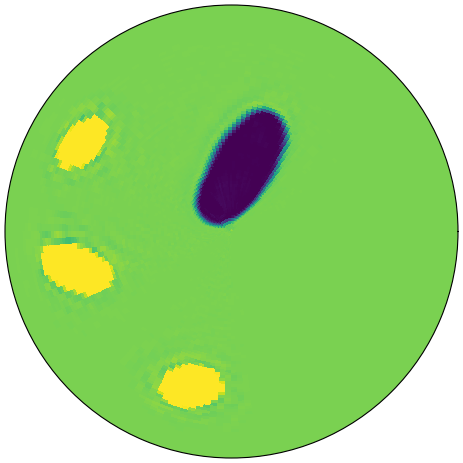} &
    \includegraphics[width=0.22\textwidth]{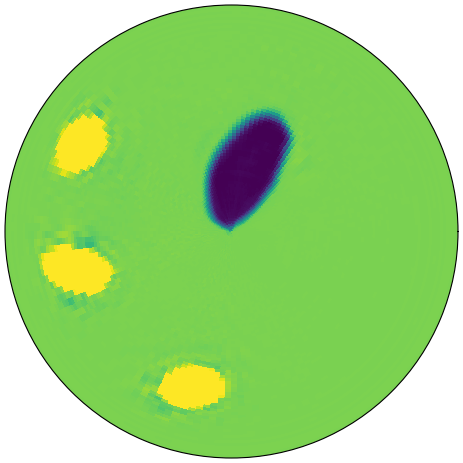} &
    \includegraphics[width=0.22\textwidth]{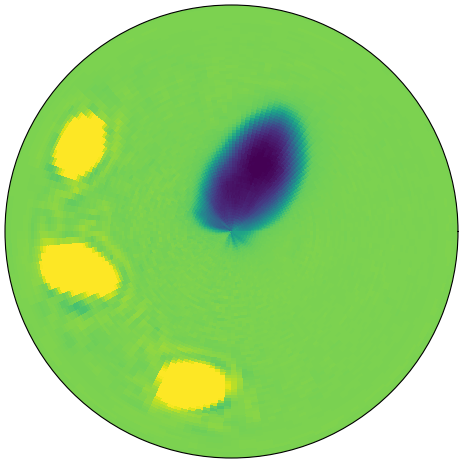}
  \end{tabular}
  \caption{\label{fig:prediction}NN prediction of a sample in the test data for negative (first row)
    / positive (second row) / mixture (third row) inclusions with $N_e=4$ for different noise level
    $\delta=0$, $2\%$ and $10\%$.
  }
\end{figure}

\begin{figure}[h!]
  \centering
  \subfigure[reference]{
    \includegraphics[width=0.22\textwidth]{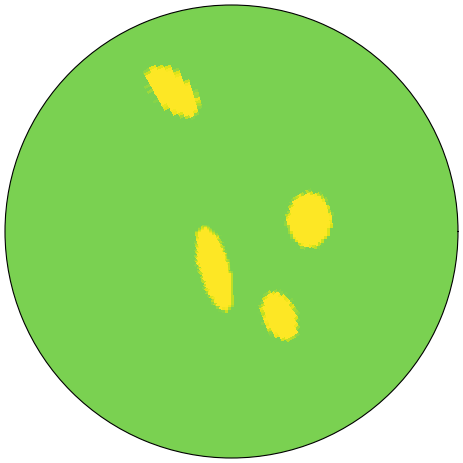}
  }
  \subfigure[$\tm^{\NN}$, $\delta=0\%$]{
    \includegraphics[width=0.22\textwidth]{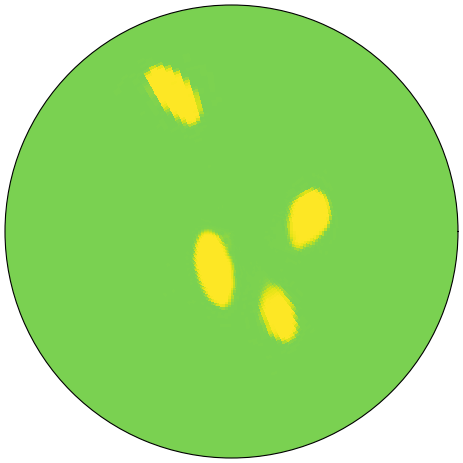}
  }
  \subfigure[$\tm^{\NN}$, $\delta=1\%$]{
    \includegraphics[width=0.22\textwidth]{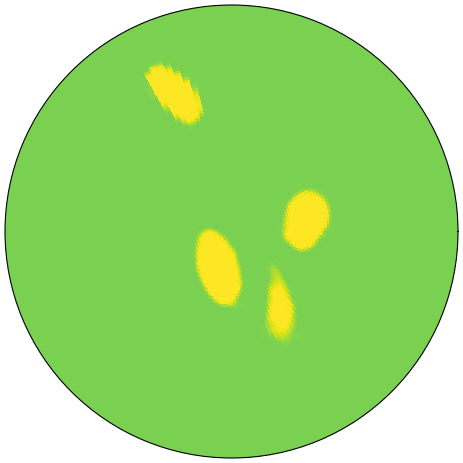}
  }
  \subfigure[$\tm^{\NN}$, $\delta=2\%$]{
    \includegraphics[width=0.22\textwidth]{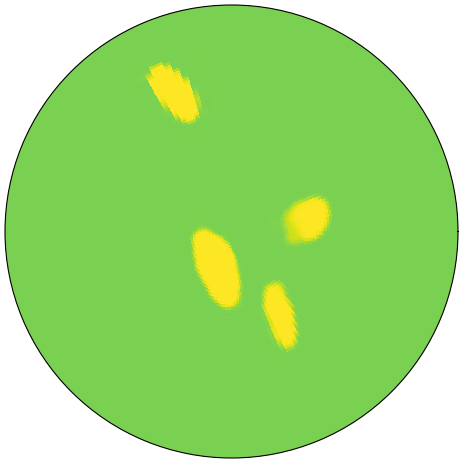}
  }

  \subfigure[reference]{
    \includegraphics[width=0.22\textwidth]{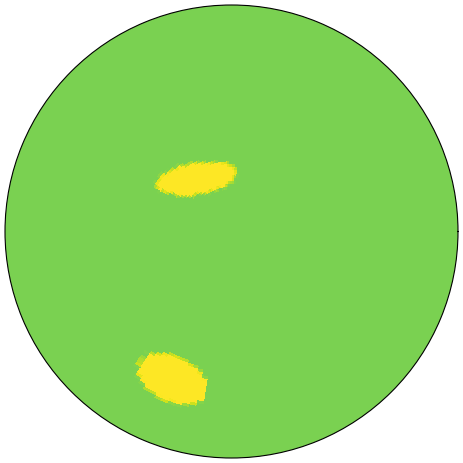}
  }
  \subfigure[$\tm^{\NN}$, $\delta=0\%$]{
    \includegraphics[width=0.22\textwidth]{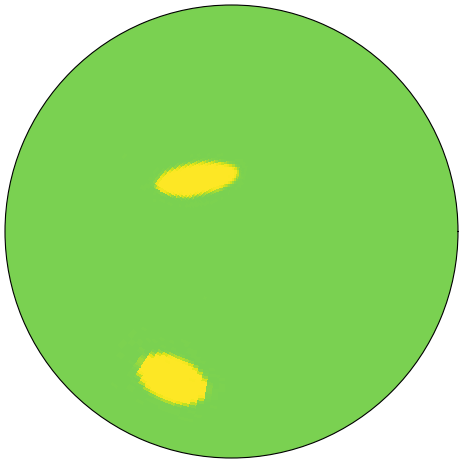}
  }
  \subfigure[$\tm^{\NN}$, $\delta=1\%$]{
    \includegraphics[width=0.22\textwidth]{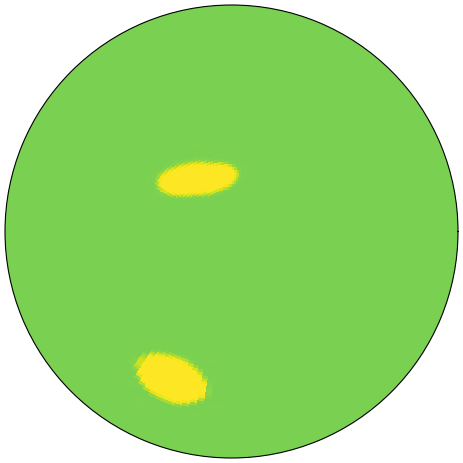}
  }
  \subfigure[$\tm^{\NN}$, $\delta=2\%$]{
    \includegraphics[width=0.22\textwidth]{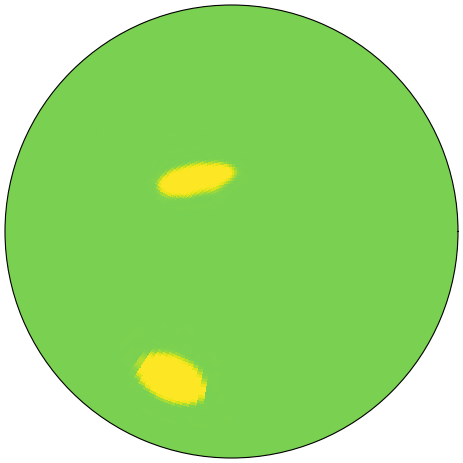}
  }
  \caption{\label{fig:generalizationPosN2N4} NN generalization test for the negative inclusions.
  The upper (or lower) figures: the NN is trained by the data of the number of ellipses $N_e=2$ (or
  $4$) with noise level $\delta=0$, $1\%$ or $2\%$ and test by the data of $N_e=4$ (or $2$) with
  the same noise level. 
      }
\end{figure}

\begin{figure}[h!]
  \centering
  \begin{tabular}{l@{}cccc}
    & reference & Neg & Pos & Mix \\
    \rotatebox{90}{\phantom{abcdefgh}Neg}\phantom{ab} &
    \includegraphics[width=0.22\textwidth]{TTPosPosN4N4Ns0Ns0Sp7Y.png} &
    \includegraphics[width=0.22\textwidth]{TTPosPosN4N4Ns0Ns0Sp7Ypred.png} & 
    \includegraphics[width=0.22\textwidth]{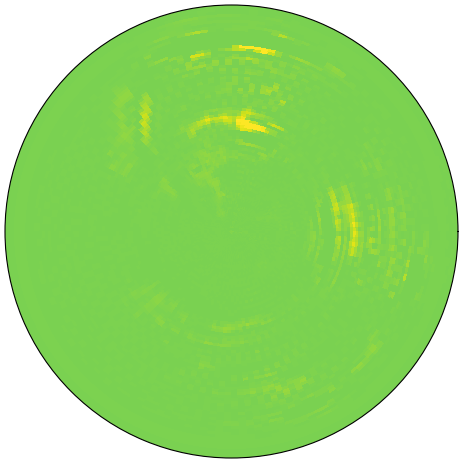} & 
    \includegraphics[width=0.22\textwidth]{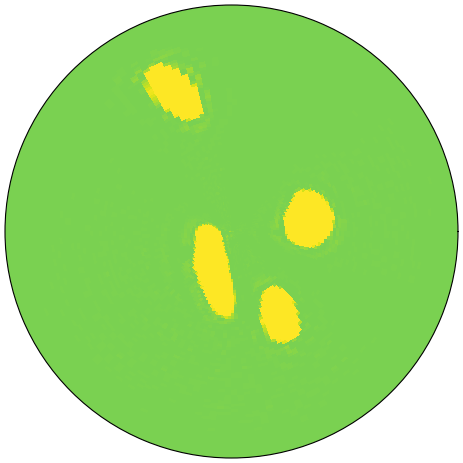} \\
    \rotatebox{90}{\phantom{abcdefgh}Pos} &
    \includegraphics[width=0.22\textwidth]{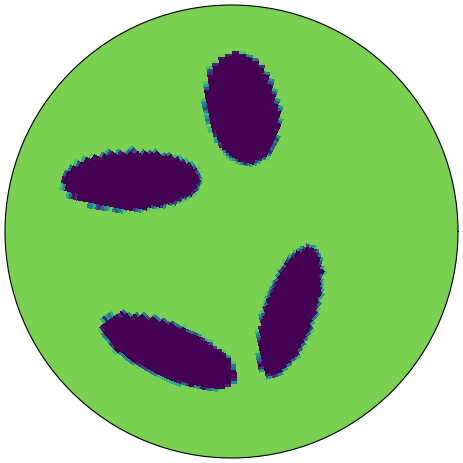} & 
    \includegraphics[width=0.22\textwidth]{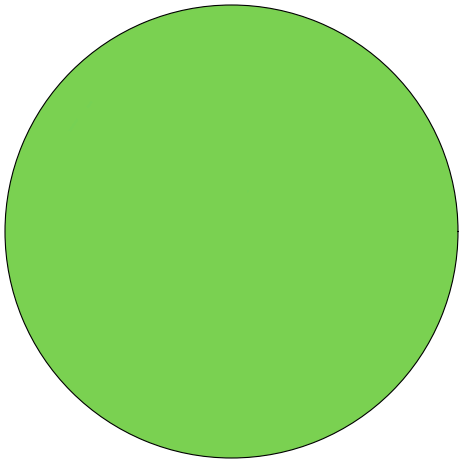} & 
    \includegraphics[width=0.22\textwidth]{TTNegNegN4N4Ns0Ns0Sp2Ypred.png} & 
    \includegraphics[width=0.22\textwidth]{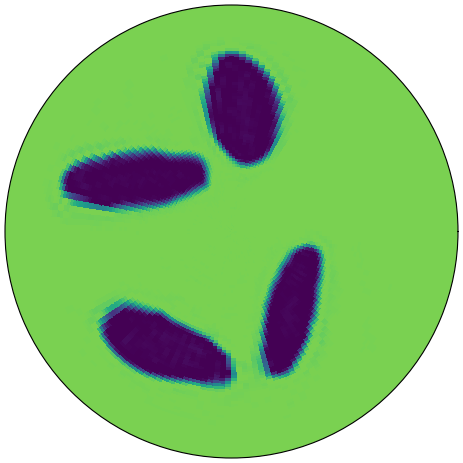} \\
    \rotatebox{90}{\phantom{abcdefgh}Mix} &
    \includegraphics[width=0.22\textwidth]{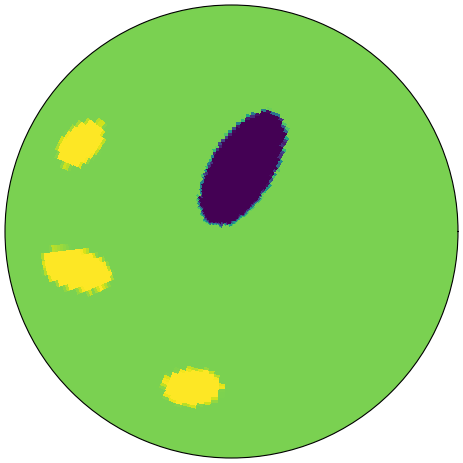} & 
    \includegraphics[width=0.22\textwidth]{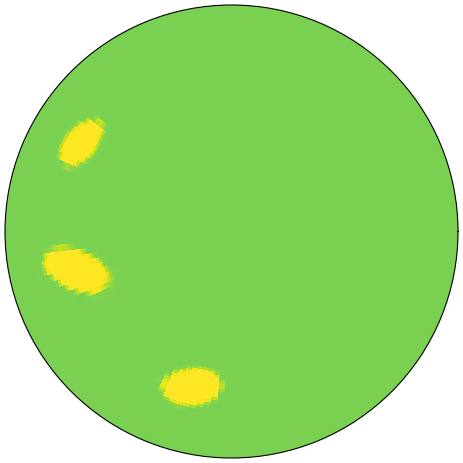} & 
    \includegraphics[width=0.22\textwidth]{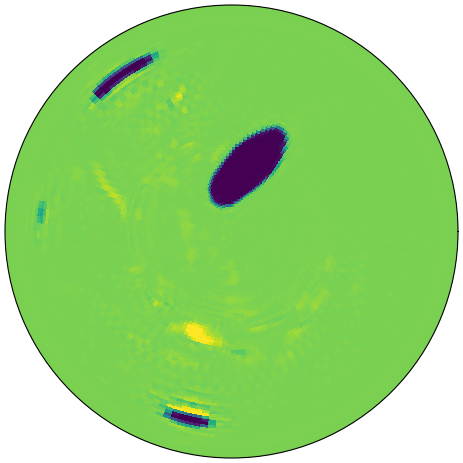} & 
    \includegraphics[width=0.22\textwidth]{TTMixMixN4N4Ns0Ns0Sp1Ypred.png}
  \end{tabular}
  \caption{\label{fig:generalization} NN generalization test for different types of data sets.
  The first column is the reference solution. In each column of the last three columns, 
  the NN is trained with one data type (negative, positive, or mixed) and is tested on all three 
  data types with $N_e=4$ and without noise.
                      }
\end{figure}

The first numerical study is concerned with the choice of channel number $c$ in
\cref{alg:inverse}. \Cref{fig:trErr} presents the test PSNR and the number of parameters with
different channel number $c$ for three types of data sets with $N_e=4$. As the channel number $c$
increases, the test PSNR first consistently increases and then saturates for all the three types of
data. Notice that the number of parameters of the neural network is
$O(c^2\log(N_{\theta})N_{\cnn})$. The choice of $c=30$ is a reasonable balance between accuracy and
efficiency and the total number of parameters is $684$K.

To model the uncertainty in the measurement data, we introduce noises to the measurement data by
defining $u^{s,\delta}(x_r)\equiv (1+ Z_i \delta)u^s(x_r)$, where $Z_i$ is a Gaussian random
variable with zero mean and unit variance and $\delta$ controls the signal-to-noise ratio. In terms
of the actual data $d$ of the differential imaging, $d^\delta(s,h) \equiv (1+Z_i\delta)d(s,h) +
Z_i\delta u^s_0(x_r)$. Notice that, since the mean of
$\frac{\|u^s_0(x_r)\|}{\|d(s,h)\|}$ for all the samples lies in $[15, 30]$ in these experiments, the
signal-to-noise ratio for $d$ is in fact more than $15\cdot\delta$. For each noisy level $\delta=0$,
$2\%$, $10\%$, an independent NN is trained and tested with the noisy data set
$\{(d_i^\delta,\tm_i)\}$.

\Cref{fig:prediction} collects, for different noise level $\delta$, samples for all three data
types: (1) negative inclusions with $N_e=4$, (2) positive inclusions with $N_e=4$, and (3) mixture
inclusions with $N_e=4$. The NN is trained with the datasets generated in the same way as the test
data. When there is no noise in the measurement data, the NN consistently gives accurate predictions
of the slowness field $\tm$, in the position, shape, and direction of the ellipse. For the small
noise levels, for example, $\delta=2\%$, the boundary of the shapes slightly blurs while the
position and direction of the ellipse are still correct. As the noise level $\delta$ increases, the
shapes become fuzzy but the position and number of shapes are always correct. This demonstrates the
proposed NN architecture is capable of learning the inverse problem.

The next test is about the generalization of the proposed NN.  We first train the NN by the data set
of the negative inclusions with $N_e=2$ (or $4$) with noise level $\delta=0$, $1\%$ or $2\%$ and
test by the data of the negative inclusions with $N_e=4$ (or $2$) with the same noise level. The
results, presented in \cref{fig:generalizationPosN2N4}, indicate that the NN trained by the data
with two inclusions is capable of recovering the measurement data of the case with four inclusions,
and vice versa.  This shows that the trained NN is capable of predicting beyond the training
scenario.

The last test is about the prediction power of the NN on one data type while trained with
another. In \cref{fig:generalization}, the first column is the reference solution. In each of the
rest three columns, the NN is trained with one data type (negative, positive, or mixed) and is
tested on all three data types, with $N_e=4$ and without noise. The figures in the second column
show that the NN trained by negative inclusions fails to capture the information of the positive
inclusions, and vice versa, the third column demonstrates that the NN trained with positive
inclusions fails for the negative inclusions.  On the other hand, the NN trained with mixed
inclusions is capable of predicting reasonably well for all three data types.

\section{Discussions}\label{sec:conclusion}
This paper presents a neural network approach for the inverse problems of first arrival traveltime
tomography, by using the NN to approximate the whole inverse map from the measurement data to the
slowness field. The perturbative analysis, which indicates that the linearized forward map can be
represented by a one-dimensional convolution with multiple channels, inspires the design of the
whole NN architectures. The analysis in this paper can also be extended to the three-dimensional TT
problems by leveraging recent work such as \citep{s.2018spherical}.

\acks{
  The work of Y.F. and L.Y. is partially supported by the U.S. Department of Energy, Office of
  Science, Office of Advanced Scientific Computing Research, Scientific Discovery through Advanced
  Computing (SciDAC) program. The work of L.Y. is also partially supported by the National Science
  Foundation under award DMS-1818449.
}

\bibliography{nn}

\appendix

\end{document}